\documentclass[a4paper]{amsart}

\usepackage{amssymb, bbm, hyperref, tikz} 
\usetikzlibrary{cd, decorations.markings, arrows} % Decorating arrows in graphs
\usepackage[capitalise, nameinlink, noabbrev, nosort]{cleveref} 

\theoremstyle{plain}
\newtheorem{lma}{Lemma}[section]
\crefname{lma}{Lemma}{Lemmata}
\newtheorem{thm}[lma]{Theorem}
\crefname{thm}{Theorem}{Theorems}
\newtheorem{cor}[lma]{Corollary}
\crefname{cor}{Corollary}{Corollaries}
\newtheorem{prp}[lma]{Proposition}
\crefname{prp}{Proposition}{Propositions}

\theoremstyle{definition}

\crefname{pgr}{Paragraph}{Paragraphs}
\newtheorem{dfn}[lma]{Definition}
\crefname{dfn}{Definition}{Definitions}
\newtheorem{ntn}[lma]{Notation}

\newcounter{theoremintro}
\newtheorem{thmintro}[theoremintro]{Theorem}

\theoremstyle{remark}
\newtheorem{rmk}[lma]{Remark}
\crefname{rmk}{Remark}{Remarks}

\crefname{rmks}{Remarks}{Remarks}
\newtheorem{exa}[lma]{Example}
\crefname{exa}{Example}{Examples}

\newcommand{\calG}{\mathcal{G}}
\newcommand{\calH}{\mathcal{H}}

\newcommand{\Bco}{\mathcal{B}_{\mathrm{c}}^{\mathrm{o}}}

\newcommand{\ran}{\mathsf{r}}
\newcommand{\sou}{\mathsf{s}}

\newcommand{\frakA}{\mathfrak{A}}

\newcommand{\Fp}{F_{\lambda}^{p}}

\newcommand{\Fpast}{F_{\lambda}^{p, \ast}}

\newcommand{\PIMP}{\mathrm{PI}_{\mathrm{MP}}}

\newcommand{\aut}{\mathbf{Aut}}
\newcommand{\inn}{\mathbf{Inn}}
\newcommand{\out}{\mathbf{Out}}
\newcommand{\Zycles}{\mathbf{Z^{1}}}
\newcommand{\Bound}{\mathbf{B^{1}}}
\newcommand{\Hquot}{\mathbf{H^{1}}}
\newcommand{\F}{\mathsf{F}}
\newcommand{\invisom}{\mathbf{U}}

\newcommand{\ad}{\mathrm{ad}}
\newcommand{\id}{\mathrm{id}}

\newcommand{\ca}{$C^*$-algebra}
\newcommand{\ba}{Banach algebra}

\newcommand{\andSep}{\,\,\,\text{ and }\,\,\,}
\newcommand{\TT}{\mathbb{T}}
\newcommand{\CC}{\mathbb{C}}

\DeclareMathOperator{\supp}{supp}

\numberwithin{equation}{section} 

\title[Invertible isometries and automorphisms of groupoid algebras]{Topological full groups, invertible isometries, and automorphisms of groupoid algebras}
\date{}

\author{Eusebio Gardella}
\author{Mathias Palmstr{\o}m}
\author{Hannes Thiel}

\address{Eusebio Gardella,
	Department of Mathematical Sciences, Chalmers University of
	Technology and University of Gothenburg, Gothenburg SE-412 96, Sweden.}
\email{gardella@chalmers.se}
\urladdr{www.math.chalmers.se/~gardella}

\address{Mathias Palmstrøm,
	Lørenskog, Norway}
\email{mathias.palmstroem@gmail.com}

\address{Hannes~Thiel, 
	Department of Mathematical Sciences, Chalmers University of Technology and the University of
	Gothenburg, SE-412 96 Gothenburg, Sweden}
\email{hannes.thiel@chalmers.se}
\urladdr{www.hannesthiel.org}

\thanks{
	EG was partially supported by the Swedish Research Council Grant 2021-04561.
	HT was partially supported by the Knut and Alice Wallenberg Foundation (KAW 2021.0140) and by the Swedish Research Council project grant 2024-04200.	
}

\subjclass[2020]%
{Primary
47L10; %algebras of operators on Banach spaces and other topological linear spaces
Secondary
46L55. %noncommutative dynamical systems
}

\begin{document}

%==========================================================================================
\begin{abstract}
We show that the topological full group of a Hausdorff ample groupoid with compact unit space coincides with the group of homotopy classes of invertible isometries in pseudofunction algebras associated with the groupoid. Moreover, if the groupoid $\calG$ is also effective, then we show that the group of (inner) automorphisms in pseudofunction algebras is a split extension of the automorphisms (respectively, the topological full group) of $\calG$ by the group of 1-cocycles (respectively, the 1-coboundaries). 
\end{abstract}	

\maketitle

\renewcommand*{\thetheoremintro}{\Alph{theoremintro}}

%==========================================================================================
%==========================================================================================
\section{Introduction} 
\label{sec: introduction}

%==========================================================================================
Étale groupoids play a central role in several branches of mathematics, including operator algebras, topological dynamics, and combinatorial group theory.
In operator algebras, they provide a rich source of interesting examples, and several Banach algebras constructed from various combinatorial or dynamical data can be realized as algebras associated with étale groupoids, offering new perspectives on their structure.
Prominent examples of Banach algebras associated with Hausdorff étale groupoids are the $p$-pseudofunctions \cite{ChoGarThi24LpRigidity, GarLup17ReprGrpdLp, BarKwa23arX:TopFreeActionsIdlsTwBAlgCrProd, BarKwaMck25BAlgTwGpdInvSgpDisintLp}, the $I$-norm completions \cite{Hah78RegReprMsrGpds}, and, more generally, the symmetrized $p$-pseudofunctions \cite{AusOrt22GpdsHermBAlg, ElkPoo23arX:PropTBAlg, Elk25SymmPseudofctLp, SamWie20QuasiHermAmen, SamWie24ExoticCAlgsGeomGps}. 
The (symmetrized) $p$-pseudofunctions include as a special instance the reduced groupoid $\rm C^*$-algebras which have been intensely studied ever since Renault's seminal monograph \cite{Ren80GrpdApproach}. 

Beyond operator algebras, étale groupoids--especially the \emph{ample} ones, whose topology admits a basis of compact open bisections--give rise to groups with remarkable algebraic properties that can be studied through their associated \emph{topological full groups} $\mathsf{F}(\mathcal{G})$.
When $\mathcal{G}$ has compact unit space, $\mathsf{F}(\mathcal{G})$ consists of the full open bisections of $\mathcal{G}$, with set multiplication and inversion as the group operations.
This construction was introduced and systematically developed by Matui \cite{Mat06RmksTopFullGpsCMS, Mat12HomTopFullGp, Mat15TopFullGpOneSidedShift}, building on earlier ideas of Krieger \cite{Kri80DimClassHomeoGps} and of Giordano, Putnam, and Skau \cite{GioPutSka99FullGpsCMS} in topological dynamics. 
A celebrated result of Jushenko and Monod \cite{JusMon13CantorSystems} showed that topological full groups arising from Cantor minimal systems provide the first examples of finitely generated, infinite, simple, and amenable groups.

Since then, topological full groups have been the focus of extensive study; see, for instance, \cite{BriSca19CSimpleReprTopFullGps, ChoJusNek20TopFullGpsZdActions, GarTan23arX:GenThompsonGpV, Sca23DichotomyTopFullGps, Tan23arX:SteinGpsTopFullGps, Nek19SimpleGpsDynOrigin, NylOrt21KEPGpdsAHConj, NylOrt21MatuiAHConjGraphGpds, Li25AmpleGpdsTopFullGpsAlgKThyInfLoopsSpaces}. 
Beyond their algebraic interest, these groups also encode deep structural information about the underlying groupoid: for several important classes of étale groupoids, $\mathsf{F}(\mathcal{G})$ forms a complete invariant \cite[Theorem~3.10]{Mat15TopFullGpOneSidedShift}, \cite[Theorem~7.2]{NylOrt19TopFullGpsAmpleGpds}, allowing one to distinguish groupoids by their topological full groups.
This phenomenon was used by Brin \cite{Bri04HigherDimThompsonGps} to show that Thompson’s group $V$ is not isomorphic to its two-dimensional analogue, $2V$.

In our recent work \cite{GarPalThi25arX:RigidityAmpleGpds}, we showed that the $I$-norm completions, $p$-pseudo\-function algebras, and symmetrized $p$-pseudofunction algebras (with $p \neq 2$) are complete invariants for ample groupoids; see \cite[Theorem~4.7, Corollary~4.8 and Corollary~4.2]{GarPalThi25arX:RigidityAmpleGpds}. 
Here we build on that framework to show that the topological full group of an ample groupoid is naturally encoded within these algebras.
Recall that for a unital \ba\ $\frakA$, an invertible element $\mathfrak{u} \in \frakA$ is called an \emph{invertible isometry} if $\| \mathfrak{u} \| = \| \mathfrak{u}^{-1} \| = 1$.
The set of all invertible isometries forms a group, denoted $\invisom(\frakA)$, and we write $\invisom(\frakA)_0$ for its connected component of the identity.

The following is our first main result; 
see \cref{prp:identifying the quotient of invertible isometries with the topological full group}.

%==========================================================================================
\begin{thmintro}
Let $\calG$ be a Hausdorff ample groupoid with compact unit space. 
Let $\frakA(\calG)$ denote either $L^I (\calG)$, $\Fpast(\calG)$, or $\Fp(\calG)$, with $p \in [1, \infty) \setminus \{2\}$. 
Then $C(\calG^{(0)},\mathbb{T})$ is the connected component of $\invisom(\frakA(\calG))$ and there is a canonical, split-exact sequence 
\[
\begin{tikzcd}
1 \arrow{r}{} & C(\calG^{(0)},\mathbb{T}) \arrow{r} & \invisom(\frakA(\calG)) \arrow{r}{\sigma} & \F(\calG) \arrow{r}{} & 1.
\end{tikzcd}
\]
In particular, there are canonical group isomorphisms
\[
\invisom(\frakA(\calG))\cong C(\calG^{(0)},\mathbb{T})\rtimes \F(\calG), \andSep
\mathsf{F}(\calG)\cong \invisom(\frakA(\calG))/\invisom(\frakA(\calG))_0.
\]  
\end{thmintro}

%==========================================================================================
We now turn to automorphism groups. 
Given a Banach algebra $\frakA$, an \emph{automorphism} of $\frakA$ will always mean an isometric algebra automorphism, and we denote the group of all such automorphisms by $\aut(\frakA)$. 
An \emph{inner automorphism} is an automorphism of the form $\ad(\mathfrak{u})$ for an invertible isometry $\mathfrak{u} \in \frakA$. 
The collection of inner automorphisms forms a group that we denote by $\inn(\frakA)$. 
We refer the reader to \cref{df:cocycles} for the 
definitions of the groups $\Zycles(\calG,\TT), \Bound(\calG,\TT)$ and $\Hquot(\calG,\TT)$. The following is our second main result; see \cref{prp:exact sequences of automorphism groups}.

%==========================================================================================
\begin{thmintro}
	Let $\calG$ be a Hausdorff ample groupoid. Then there is a canonical, split-exact sequence 
	\[
	\begin{tikzcd}
		1 \arrow{r}{} & \Zycles (\calG, \TT) \arrow{r}{\Gamma} & \aut(\frakA(\calG)) \arrow{r}{\Omega} & \aut(\calG) \arrow{r}{} & 1. 
	\end{tikzcd}
	\]
If, in addition, $\calG$ is effective and has compact unit space, then there are induced split exact sequences
	\[
	\begin{tikzcd}
		1 \arrow{r}{} & \Bound (\calG, \TT) \arrow{r}{\Gamma} & \inn(\frakA(\calG)) \arrow{r}{\Omega} & \F(\calG) \arrow{r}{} & 1 ,
	\end{tikzcd}
	\]
	and 
	\[
	\begin{tikzcd}
		1 \arrow{r}{} & \Hquot (\calG, \TT) \arrow{r}{\overline{\Gamma}} & \out(\frakA(\calG)) \arrow{r}{\overline{\Omega}} & \aut(\calG)/ \F(\calG) \arrow{r}{} & 1.
	\end{tikzcd}
	\]
\end{thmintro}

Our methods are very conceptual and allow us to handle all the algebras $L^I (\calG)$, $\Fpast(\calG)$, and $\Fp(\calG)$, with $p \in [1, \infty) \setminus \{2\}$. Indeed,
the arguments developed here rely on a careful identification of the compact, open bisections of an ample groupoid with the spatial partial isometries of $\frakA(\calG)$, which is possible when $\frakA(\calG)$ is any of the algebras mentioned above.

%==========================================================================================
%==========================================================================================
\section{Topological full groups and invertible isometries} 

%==========================================================================================
Throughout, we shall freely use the notation of \cite{GarPalThi25arX:RigidityAmpleGpds}, and unless otherwise specified, $\calG$ will always denote a Hausdorff étale groupoid with compact unit space. 
Our goal in this section is to obtain a generalization, to the setting of Hausdorff \'etale groupoids, of \cite[Corollary~4.16]{GarThi22IsoConv}, where it is shown that if $G$ is a discrete group and $p\neq 2$, then there are natural group isomorphisms
\[\invisom(F^p_\lambda(G))\cong \mathbb{T}\times G \ \ \mbox{ and } \ \ G\cong \pi_0(\invisom(F^p_\lambda(G))).\]

We begin by introducing some standard terminology.

%==========================================================================================
\begin{dfn} 
\label{dfn: topological full groups}	
Let $\calG$ be a Hausdorff étale groupoid with compact unit space. The \emph{topological full group} $\F(\calG)$ is the group of open bisections $B \subseteq \calG$ satisfying $\ran(B) = \sou(B) = \calG^{(0)}$, with multiplication given by
\[
AB = \big\{ ab \in \mathcal{G} \colon a \in A, b \in B \mbox{ and } \sou(a) = \ran(b) \big\},
\]
and inversion $A^{-1} = \{a^{-1}\in \mathcal{G} \colon a \in A \}$, for $A, B \in \F(\calG)$.
	
Every element $B \in \F(\calG)$ determines a homeomorphism $\rho_B$ of $\calG^{(0)}$ given by $$\rho_B (u) = \ran(Bu) = \big(\ran|_{B} \circ \sou|_{B}^{-1}\big) (u) ,$$ for $u \in \calG^{(0)}$. 
Thus, we obtain a group homomorphism $\rho\colon \F(\calG) \to \mathrm{Homeo}(\calG^{(0)})$, which is injective if $\calG$ is effective (see \cite{Nek19SimpleGpsDynOrigin}). 
Through this homomorphism, we regard $\F(\calG)$ as acting on $\calG^{(0)}$ via homeomorphisms.
\end{dfn}

%==========================================================================================
\begin{rmk} 
\label{rmk: explanation of different terminology for TFGs}
In \cite{Mat12HomTopFullGp}, Matui defined the topological full group of $\calG$ as the image of $\F(\calG)$ under the above homomorphism, but here we use the definition of Nekrashevych from \cite[Definition 2.6]{Nek19SimpleGpsDynOrigin} which is more suitable for our purposes. However, \cref{dfn: topological full groups} coincides with Matui's definition when the étale groupoid is effective.
\end{rmk}

%==========================================================================================
Notice that when $\calG$ has compact unit space and $B \in \F(\calG)$, then $B$ is not only open, but also compact, hence the indicator function $\mathbbm{1}_B$ is continuous and supported on the compact open bisection $B$, and it satisfies 
\[
1 = \|\mathbbm{1}_B\|_\infty = \|\mathbbm{1}_B\|_{I}.
\]

By $\frakA(\calG)$ we shall mean any one of the Banach algebras $L^I (\calG)$, $\Fpast(\calG)$, or $\Fp(\calG)$, with $p \in [1, \infty) \setminus \{2\}$; see \cite[Paragraph~3.2, Paragraph~3.3 and Definition~4.3]{GarPalThi25arX:RigidityAmpleGpds}, respectively, for their definitions. 
We have, 
\[ 
\mathbbm{1}_{\calG^{(0)}} = \mathbbm{1}_B \mathbbm{1}_{B^{-1}} = \mathbbm{1}_{B^{-1}} \mathbbm{1}_{B},
\] 
so that $\mathbbm{1}_B \in \invisom(\frakA(\calG))$. In addition, if $f \in C(\calG^{(0)} , \TT)$, then $f \ast \mathbbm{1}_B$ belongs to $\invisom(\frakA(\calG))$ and its inverse is given by \[
\mathbbm{1}_{B^{-1}} \ast \overline{f}=(\overline{f} \circ\rho_B)\ast\mathbbm{1}_{B^{-1}} \in \invisom(\frakA(\calG)).
\] 

In fact, all invertible isometries in $\frakA(\calG)$ are of this form; we isolate this fact for future use, together with other convenient descriptions of these elements.
We will write $\invisom(\frakA(\calG))_0$ for the invertible isometries which are homotopic to the identity within $\invisom(\frakA(\calG))$. 

%==========================================================================================
\begin{prp} 
\label{prp:UFpGpd}
Let $\calG$ be a Hausdorff étale groupoid with compact unit space. Then
\[
\invisom(\frakA(\calG)) = \big\{ \mathfrak{a}\in \frakA(\calG) \colon \supp(\mathfrak{a}) \in \F(\calG) \text{ and } \mathfrak{a} \text{ is } \TT\text{-valued on } \supp(\mathfrak{a}) \big\}.
\]
Moreover, for every $\mathfrak{a}\in \invisom(\frakA(\calG))$ there are unique
$f\in C(\calG^{(0)},\mathbb{T})$ and $B\in \F(\calG)$ such that $\mathfrak{a}=f\ast \mathbbm{1}_B$.

Finally, if $\calG$ is ample, then $\invisom(\frakA(\calG))_0=\invisom(C(\calG^{(0)}))=C(\calG^{(0)},\mathbb{T})$.
\end{prp}
\begin{proof}
By \cite[Proposition~4.6]{GarPalThi25arX:RigidityAmpleGpds} and \cite[Proposition~2.21]{GarPalThi25arX:RigidityAmpleGpds}, hermitian idempotents in $\frakA(\calG)$ are ultrahermitian and commute (see \cite[Paragraph~2.3 and Definition~2.15]{GarPalThi25arX:RigidityAmpleGpds}). 
Thus, it follows from \cite[Theorem~2.17]{GarPalThi25arX:RigidityAmpleGpds} that the Moore-Penrose partial isometries, $\PIMP(\frakA(\calG))$, in $\frakA(\calG)$ form an inverse semigroup under multiplication. 
It is easy to see that $\invisom(\frakA(\calG))$ is a subgroup of $\PIMP(\frakA(\calG))$ with $\mathfrak{a}^\dag = \mathfrak{a}^{-1}$, for all $\mathfrak{a} \in \invisom(\frakA(\calG))$.  For the converse inclusion, it follows from the description of the elements in $\PIMP(\frakA(\calG))$ provided by \cite[Proposition~3.11]{GarPalThi25arX:RigidityAmpleGpds} that if $\mathfrak{a} \in \invisom(\frakA(\calG))$, then $\supp(\mathfrak{a})$ belongs to $\F(\calG)$ and $\mathfrak{a}$ is $\TT$-valued on its support. This proves the displayed identity in 
the proposition. 

For the second assertion, let  $\mathfrak{a}\in \invisom(\frakA(\calG))$ be given.
Define $f_{\mathfrak{a}}\in C(\calG^{(0)},\mathbb{T})$ by $f_\mathfrak{a} (u) = \mathfrak{a}(\ran|_{\supp(\mathfrak{a})}^{-1}(u))$ for all $u \in \calG^{(0)}$.
We claim that $\mathfrak{a} = f_\mathfrak{a} \ast \mathbbm{1}_{\supp(\mathfrak{a})}$.
In order to check this identity, we regard both elements as functions $\calG\to\mathbb{C}$ and check that they agree pointwise. Given $x\in\calG$, we have 
\begin{align*}
\big(f_{\mathfrak{a}}\ast \mathbbm{1}_{\supp(\mathfrak{a})}\big)(x)
&= \sum_{yz=x} f_{\mathfrak{a}}(y) \mathbbm{1}_{\supp(\mathfrak{a})}(z).
\end{align*}
Since $f_{\mathfrak{a}}$ is supported on $\calG^{(0)}$, the above expression is zero 
unless $y$ belongs to $\calG^{(0)}$, which forces $y=\mathsf{r}(x)$ and $z=x$. 
Thus, the sum reduces to zero if $x\notin\supp(\mathfrak{a})$, and $\mathfrak{a}(y\supp(\mathfrak{a}))=\mathfrak{a}(x)$ otherwise. 
In other words, $\big(f_{\mathfrak{a}}\ast \mathbbm{1}_{\supp(\mathfrak{a})}\big)(x)=\mathfrak{a}(x)$, proving the claim. 

It follows from the above that every invertible isometry in $\frakA(\calG)$ has the form $f\ast \mathbbm{1}_B$ for some $f\in C(\calG^{(0)},\mathbb{T})$ and some $B\in \F(\calG)$.
Uniqueness of $B$ follows from the fact that $B=\supp(f\ast \mathbbm{1}_B)$, while uniqueness of $f$ follows from the fact that $\mathbbm{1}_B \mathbbm{1}_{B^{-1}}=\mathbbm{1}_{\calG^{(0)}}$.

The last part of the statement, in the case that $\calG$ is ample, follows from the first part together with \cite[Lemma~3.12]{GarPalThi25arX:RigidityAmpleGpds}.
\end{proof}

%==========================================================================================
We want to use the above characterization of invertible isometries in pseudofunction algebras to show that $\mathsf{F}(\calG)$ can be obtained, in a canonical way, from $\frakA(\calG)$. We need some preparation first.

%==========================================================================================
\begin{dfn}
Let $\calG$ be a Hausdorff ample groupoid with compact unit space. 
We define the support map $\sigma\colon  \invisom(\frakA(\calG))\to \F(\calG)$ by setting $\sigma(\mathfrak{a})=\supp(\mathfrak{a})$ for all $\mathfrak{a}\in\frakA(\calG)$.
\end{dfn}

%==========================================================================================
Note that the range of $\sigma$ is really contained in $\F(\calG)$ by \cref{prp:UFpGpd}. 

%==========================================================================================
\begin{lma}
\label{lma:sigma}
Let $\calG$ be a Hausdorff ample groupoid with compact unit space. 
Then $\sigma\colon  \invisom(\frakA(\calG))\to \F(\calG)$ is a group homomorphism.\end{lma}
\begin{proof}
It is clear that $\sigma(1)=\calG^{(0)}$, which is the unit of $\F(\calG)$, so it suffices to check that $\sigma$ is multiplicative.
Let $f,g\in C(\calG^{(0)},\mathbb{T})$ and $A,B\in\F(\calG)$ be given, and set $\mathfrak{a}=f\ast \mathbbm{1}_A$ and $\mathfrak{b}=g\ast \mathbbm{1}_B$.
Then
\begin{align*}
	\mathfrak{a}\ast \mathfrak{b}=(f \ast \mathbbm{1}_A ) \ast (g \ast \mathbbm{1}_B) &= f \ast (\mathbbm{1}_A \ast g \ast \mathbbm{1}_{A^{-1}}) \ast \mathbbm{1}_A \ast \mathbbm{1}_B \\
&= (f \cdot (g \circ \rho_{A^{-1}})) \ast \mathbbm{1}_{AB}.\end{align*}
Since $f \cdot (g \circ \rho_{A^{-1}})$ belongs to $C(\calG^{(0)},\mathbb{T})$, the uniqueness part of \cref{prp:UFpGpd} implies that $\supp(\mathfrak{a}\ast \mathfrak{b})=AB$, as desired, so that $\sigma(\mathfrak{a}\ast \mathfrak{b})=\sigma(a)\sigma(b)$. 
\end{proof}

%==========================================================================================
The following is the first main result in this work.

%==========================================================================================
\begin{thm} 
\label{prp:identifying the quotient of invertible isometries with the topological full group}
Let $\calG$ be a Hausdorff ample groupoid with compact unit space. 
Let~$\frakA(\calG)$ denote either $L^I (\calG)$, $\Fpast(\calG)$, or $\Fp(\calG)$, with $p \in [1, \infty) \setminus \{2\}$. Then the sequence 
\[
\begin{tikzcd}
	1 \arrow{r}{} & \invisom(\frakA(\calG))_0 \arrow{r}{\iota} & \invisom(\frakA(\calG)) \arrow{r}{\sigma} & \F(\calG) \arrow{r}{} & 1 ,
\end{tikzcd}
\]
is split exact.
In particular, there is a canonical identification
\[
\invisom(\frakA(\calG))\cong C(\calG^{(0)},\mathbb{T})\rtimes \F(\calG),
\] 
where the action of $\F(\calG)$ on $C(\calG^{(0)},\mathbb{T})$ is given by $\rho$ (see \cref{dfn: topological full groups}). 
Moreover, there is a canonical group isomorphism
\[
\Psi\colon \F(\calG) 
\to \pi_0\big(\invisom(\frakA(\calG))\big)
= \invisom(\frakA(\calG)) / \invisom_0(\frakA(\calG))
\]
given by $\Psi(B)=[\mathbbm{1}_B]_0$ for all $B\in\F(\calG)$.
\end{thm}
\begin{proof}
We begin by establishing the following claim ($\simeq$ is the equivalence relation of homotopy). 

\textbf{Claim:} \emph{Let $\mathfrak{a}_0 , \mathfrak{a}_1 \in \invisom(\frakA(\calG))$ be given. Then $\mathfrak{a}_0 \simeq \mathfrak{a}_1$ in $\PIMP(\frakA(\calG))$ if and only if $\mathfrak{a}_0 \simeq \mathfrak{a}_1$ in $\invisom(\frakA(\calG))$.} The ``if" implication is obvious. For the converse, let $\mathfrak{a}_t$ implement a homotopy in $\PIMP(\frakA(\calG))$ between $\mathfrak{a}_0$ and $\mathfrak{a}_1$. By compactness, there exist $0 = t_0 < t_1 < \ldots < t_{n-1} < t_n = 1$ such that $\|\mathfrak{a}_{t} - \mathfrak{a}_{t^{\prime}}\| < 1$ whenever $t,t^\prime \in [t_i  , t_{i+1}]$, for $i = 0,1, \ldots n-1$. In particular, $\|\mathfrak{a}_0 - \mathfrak{a}_{t}\| < 1$ for all $t \in [0,t_1]$, which implies that $\mathfrak{a}_{t}$ is invertible. Since $\mathfrak{a}_{t} = \mathfrak{a}_{t} \mathfrak{a}_{t}^{\dag} \mathfrak{a}_{t}$, we deduce that $\mathfrak{a}_{t}^{-1} = \mathfrak{a}_{t}^{\dag}$, so that $\mathfrak{a}_{t} \in \invisom(\frakA(\calG))$ for all $t \in [0,t_1]$. By induction, we get that $\mathfrak{a}_t \in \invisom(\frakA(\calG))$ for all $t \in [0,1]$, and the claim follows.

We turn to showing that the sequence in the statement is split exact.
We will identify $\invisom(\frakA(\calG))_0$ canonically with $C(\calG^{(0)},\mathbb{T})$ via \cref{prp:UFpGpd}. 
It is immediate that $\sigma\circ\iota=\calG^{(0)}$, since the support of any function
in $C(\calG^{(0)},\mathbb{T})$ is $\calG^{(0)}$. Conversely, it follows from the description of $\invisom(\frakA(\calG))$ given in \cref{prp:UFpGpd} together with the claim just proved and \cite[Lemma 3.12]{GarPalThi25arX:RigidityAmpleGpds}, that 
$\sigma(\mathfrak{a})=\calG^{(0)}$ if and only if $\mathfrak{a}\in\invisom(\frakA(\calG))_0$, showing that the sequence is exact.
Finally, a section for $\sigma$ is given by $B\mapsto \mathbbm{1}_B$, for $B\in \F(\calG)$.
\end{proof}
	
%==========================================================================================
\begin{rmk}
The split exactness of the above sequence has been shown independently in \cite{GarGun24pre:EmbLpOpAlg}, with different methods, for étale groupoids that are effective.
\end{rmk}
	
%==========================================================================================
We now use the above to deduce some interesting consequences. We always consider
topological full groups as discrete groups.

%==========================================================================================
\begin{cor}
\label{cor:Amenability}
Let $\calG$ be a Hausdorff ample groupoid with compact unit space.
Then the following are equivalent:
\begin{enumerate} 
\item 
$\invisom(\frakA(\calG))$ is amenable in the discrete topology.
\item 
$\invisom(\frakA(\calG))$ is amenable in the norm topology.
\item 
$\F(\calG)$ is amenable. 
\end{enumerate}
\end{cor}
\begin{proof}
Recall that a topological group~$G$ is said to be amenable if there exists a left-invariant mean on the space $\mathrm{RUCB}(G)$ of right-uniformly-continuous, bounded functions on~$G$.
If $G$ is amenable in the discrete topology, then there exists a left-invariant mean on $\ell^\infty(G)$;
restricting it to $\mathrm{RUCB}(G) \subseteq \ell^\infty(G)$ we deduce amenability of $G$.
This proves that~(1) implies~(2), in fact for any topological group.

To show that~(2) implies~(3), assume that $\invisom(\frakA(\calG))$ is amenable in the norm topology.
Since $\invisom(\frakA(\calG))_0$ is an open, normal subgroup of $\invisom(\frakA(\calG))$, the quotient topology on $\invisom(\frakA(\calG)) / \invisom(\frakA(\calG))_0$ induced by the norm topology on $\invisom(\frakA(\calG))$ is discrete.
By \cite[Theorem~4.6]{Ric67AmenGpsFixedPtPty}, amenability passes to quotients by closed, normal subgroups.
Hence, the quotient $\invisom(\frakA(\calG)) / \invisom(\frakA(\calG))_0$  is amenable (as a discrete group), and therefore $\F(\calG)$ is amenable via the isomorphism from 
%\Psi\colon \F(\calG) = \invisom(\frakA(\calG)) / \invisom_0(\frakA(\calG))
\cref{prp:identifying the quotient of invertible isometries with the topological full group}.

To show that~(3) implies~(1), assume that $\F(\calG)$ is amenable. 
By \cref{prp:UFpGpd}, the group $\invisom(\frakA(\calG))_0$ is abelian and thus amenable in the discrete topology. 
Further, the quotient $\invisom(\frakA(\calG)) / \invisom_0(\frakA(\calG))$ is isomorphic to $\F(\calG)$, and therefore amenable by assumption.
Using that a discrete group $G$ is amenable whenever it contains a normal subgroup $N$ such that both~$N$ and $G/N$ are amenable, it follows that $\invisom(\frakA(\calG))$ is amenable in the discrete topology.
\end{proof}

%==========================================================================================
Little appears to be known in general about the group of invertible isometries in Banach algebras, while there is a growing amount of knowledge about topological full groups. 
\cref{cor:Amenability} allows us to establish some properties of the group of invertible isometries. Specifically for $L^p$-Cuntz algebras, we can 
deduce nonamenability of their invertible isometry groups.

%==========================================================================================
\begin{exa}
For $n\geq 2$, let $\calG_n$ denote the Cuntz groupoid, which is the groupoid associated with the graph with one vertex and $n$ loops. Then $F^p(\calG_n)$ is the $L^p$-Cuntz algebra $\mathcal{O}_{n}^{p}$ introduced in \cite{Phi12arX:LpAnalogsCtz}; see \cite{GarLup17ReprGrpdLp}. Moreover, the topological full group $\F(\calG_n)$ of $\calG_n$ is well-known to be Thompson's group $V_n$; see \cite{Nek04CuntzPimsnerAlgsGpActions}. 
%When $p=2$, it is known that $\invisom(\mathcal{O}_{n}^{p})$ is amenable, while 
For $p\neq 2$, nonamenability of $V_n$ together with \cref{cor:Amenability} implies that $\invisom(\mathcal{O}_{n}^{p})$ is non-amenable. The same result holds for $p=2$ (although the proof is very different), which is a consequence of \cite[Theorem~3.7]{Ng_amenability_2006}, since $\mathcal{O}_2$ is well-known not to be stably finite. 
\end{exa}
	
%==========================================================================================
In the next corollary, we are able to recover the rigidity obtained in \cite{GarPalThi25arX:RigidityAmpleGpds} for minimal, effective groupoids.
	
%==========================================================================================
\begin{cor} 
\label{rmk: remark on how the rigidity result already follows from the identification of the topological topological full group}
Let $\calG$ and $\calH$ be Hausdorff, ample, minimal, effective groupoids whose unit spaces are Cantor sets. Then there is an isometric isomorphism $\frakA (\calG) \cong \frakA(\calH)$ if and only if $\calG\cong \calH$.
\end{cor}
\begin{proof}
Since the ``if" implication is obvious, we prove the converse. Let $\varphi\colon \frakA (\calG) \to\frakA(\calH)$ be an isometric isomorphism. It is straightforward to see that $\varphi(\invisom(\frakA(\calG)))= \invisom(\frakA(\calH))$, and that the same is true for the connected components of the identity. 
Moreover, $\varphi$ induces an isomorphism 
\[
\invisom(\frakA(\calG))/\invisom(\frakA(\calG))_0 \cong \invisom(\frakA(\calH))/\invisom(\frakA(\calH))_0,
\] 
and thus it follows from \cref{prp:identifying the quotient of invertible isometries with the topological full group} that
\[
\F(\calG) \cong \invisom(\frakA(\calG))/\invisom(\frakA(\calG))_0 \cong \invisom(\frakA(\calH))/\invisom(\frakA(\calH))_0 \cong \F(\calH).
\]
An application of \cite[Theorem 3.10]{Mat15TopFullGpOneSidedShift} then gives that $\calG \cong \calH$, as desired. 
\end{proof}

%==========================================================================================
%==========================================================================================
\section{Automorphism groups}

%==========================================================================================
In view of the rigidity results from \cite{GarPalThi25arX:RigidityAmpleGpds}, it is natural to expect a close relationship between the automorphism groups of $\calG$ and of $\frakA(\calG)$, at least when $\calG$ is a Hausdorff, ample groupoid. 
The goal of this section is to make this relationship concrete and explicit. 
For this, we need to establish some notation and definitions first.

%==========================================================================================
\begin{ntn}
For a Banach algebra $\frakA$, we will denote by
$\aut(\frakA)$ its automorphism group. When $\frakA$ is unital, we denote by $\inn(\frakA)$ the (normal) subgroup of $\aut(\frakA)$ consisting of those 
automorphisms of the form $\mathrm{Ad}(\mathfrak{u})$, for $\mathfrak{u} \in \invisom(\frakA)$; we call these the \emph{inner} automorphisms of $\frakA$. In this case, we write $\out(\frakA)$ for the outer automorphism group of $\frakA$, namely the quotient $\out(\frakA)= \aut(\frakA) / \inn (\frakA)$. 

If $\calG$ is a groupoid, we will denote by $\aut(\calG)$ its group of automorphisms.
\end{ntn}

%==========================================================================================
Next, we introduce cocycles on ample groupoids in a way that is convenient to us; we refer the reader to \cite[Definition~I.1.12]{Ren80GrpdApproach} for the general definition.
Given $f\in C(\calG^{(0)} , \TT)$, we let $\widehat{f}\colon \calG\to \TT$ be given by $\widehat{f}(x)=\overline{f(\mathsf{r}(x))} f(\mathsf{s}(x))$ for all $x \in \calG$.

%==========================================================================================
\begin{dfn}
\label{df:cocycles}
% Let $\calG$ be a Hausdorff ample groupoid. Regard $\TT$ as a group (and therefore as a groupoid with one-point unit space).
A \emph{1-cocycle} on $\calG$ is a continuous groupoid homomorphism $\calG\to \TT$, 
and we denote by $\Zycles (\calG, \TT)$ the group of all 1-cocycles endowed with the pointwise product. 
A 1-cocycle $\xi\colon \calG\to\TT$ is said to be a \emph{1-coboundary} if there exists $f \in C(\calG^{(0)} , \TT)$ such that $\xi=\widehat{f}$. 
We write $\Bound (\calG, \TT)$ for the (normal) subgroup of $\Zycles(\calG,\TT)$ consisting
of 1-coboundaries. 
Finally, the \emph{1-cohomology} group of $\calG$ is 
\[
\Hquot (\calG, \TT) = \Zycles (\calG, \TT) / \Bound (\calG, \TT).
\] 
\end{dfn}

%==========================================================================================
We now proceed to define homomorphisms between the groups introduced above. 

%==========================================================================================
\begin{dfn}
\label{dfn:Gamma}
Let $\calG$ be a Hausdorff ample groupoid. 
Given $\xi \in \Zycles (\calG, \TT)$, we define 
a map $\Gamma(\xi)\colon C_c(\calG)\to C_c(\calG)$ by setting
\[\Gamma (\xi) (f) (x) = \xi (x) f(x)\]
for all $f \in C_c (\calG)$ and $x \in \calG$. It follows easily from the fact that $\xi$ is a continuous groupoid homomorphism that $\Gamma(\xi)$ extends to an automorphism of $\frakA(\calG)$, which we continue to denote by $\Gamma(\xi)$. 
\end{dfn}

%==========================================================================================
The assignment described in the above definition determines a map
\[
\Gamma\colon \Zycles(\calG,\TT)\to \aut(\frakA(\calG)),
\] 
which is readily seen to be a group homomorphism.
The construction of the homomorphism $\Omega\colon \aut(\frakA(\calG))\to \aut(\calG)$ takes a bit more work. 

%==========================================================================================
\begin{ntn}
\label{dfn:Phi}
Let $\calG$ be a Hausdorff ample groupoid and let $\mathcal{B}^{\mathrm{o}}_\mathrm{c}(\calG)$ denote the inverse semigroup of its compact, open bisections. Recall that $\mathbbm{1}_B$ belongs to $\PIMP(\frakA(\calG))$ for all $B\in\mathcal{B}^{\mathrm{o}}_\mathrm{c}(\calG)$. Let $\sigma\colon  \PIMP(\frakA(\calG))\to \mathcal{B}^{\mathrm{o}}_\mathrm{c}(\calG)$ denote the inverse-semigroup homomorphism given by $\sigma(\mathfrak{a})=\mathrm{supp}(\mathfrak{a})$ for all $\mathfrak{a}\in\PIMP(\frakA(\calG))$. We define a group homomorphism
\[\Phi\colon \aut(\frakA(\calG))\to \aut(\mathcal{B}^{\mathrm{o}}_\mathrm{c}(\calG))\] by
setting
\[\Phi_\alpha(B)=\sigma(\alpha(\mathbbm{1}_B))\] for all $\alpha\in\aut(\frakA(\calG))$ and all $B\in \mathcal{B}^{\mathrm{o}}_\mathrm{c}(\calG)$. 

Recall that $C_0(\calG^{(0)})$ is the C*-core 
of $\frakA(\calG)$ (see either \cite{ChoGarThi24LpRigidity} or \cite{GarPalThi25arX:RigidityAmpleGpds}), and is therefore intrinsically encoded in $\frakA(\calG)$. Thus, any automorphism of $\frakA(\calG)$ restricts to an automorphism of $C_0(\calG^{(0)})$, which corresponds via Gelfand duality with a homeomorphism of $\calG^{(0)}$. In other words, there is a canonical group homomorphism
\[\Upsilon\colon  \aut(\frakA(\calG))\to \mathrm{Homeo}(\calG^{(0)})\]  
satisfying $f(\Upsilon_\alpha^{-1}(u))=\alpha(f)(u)$ for all $f\in C_0(\calG^{(0)})$, for all $\alpha\in\aut(\frakA(\calG))$ and all $u\in \calG^{(0)}$.
\end{ntn}

%==========================================================================================
The homomorphisms $\Phi$ and $\Upsilon$ are compatible in the following sense.

%==========================================================================================
\begin{lma} \label{lma: compatibility of the induced homomorphisms from an automorphism}
Let $\calG$ be a Hausdorff ample groupoid, let $\alpha\in\aut(\frakA(\calG))$ and let $B \in \Bco(\calG)$. 
Then 
\begin{equation} 
\label{eq: compatibility of homeomorpism and automomorphism of inverse semigroup}
	\Upsilon_\alpha(\sou(B)) = \sou(\Phi_\alpha(B))  \ \ \mbox{ and } \ \ \Upsilon_\alpha(\ran(B)) = \ran(\Phi_\alpha(B)).
\end{equation}
\end{lma}
\begin{proof}
For the first identity, note that
$$ \mathbbm{1}_{\Upsilon_\alpha(\sou(B))} = \mathbbm{1}_{\sou(B)} \circ \Upsilon_\alpha^{-1} = \alpha(\mathbbm{1}_{\sou(B)}) = \alpha(\mathbbm{1}_{B^{-1}} \mathbbm{1}_{B}) = \alpha(\mathbbm{1}_{B})^{-1} \alpha(\mathbbm{1}_{B}) = \mathbbm{1}_{\sou(\Phi_\alpha(B))}.$$
The second identity follows similarly. 
\end{proof}

%==========================================================================================
\begin{prp}
\label{prop:Omega}
Let $\calG$ be an ample groupoid and let $\alpha\in\aut(\frakA(\calG))$. Given $x\in\calG$ and a clopen bisection $B\in \Bco(\calG)$ containing $x$, the expression 
\[\Omega_\alpha(x)= \Upsilon_\alpha(\mathsf{r}(x)) \Phi_\alpha(B) = \Phi_\alpha(B)\Upsilon_\alpha(\mathsf{s}(x))\in \calG,\]
is independent of $B$. Moreover, $\Omega_\alpha$ is an automorphism of $\calG$, and the resulting assignment
$\Omega\colon \aut(\frakA(\calG))\to\aut(\calG)$ is a group homomorphism.
\end{prp}
\begin{proof}
We divide the proof into claims. For the first one, we recall
that if $A,B\subseteq \calG$ are compact, open bisections, then $A\subseteq B$ if and only if $A = B A^{-1} A$.

\medskip

\textbf{Claim 1:} \emph{$\Omega_\alpha$ is well-defined.}
Let $x\in \calG$, and let $$\mathcal{B}_x = \{B \in \Bco(\calG) \colon x \in B\}.$$ Since $\Phi_\alpha$ is an isomorphism, it follows by what we have just recalled that given $A, B \in \Bco(\calG)$, we have $A \subseteq B$ if and only if $\Phi_\alpha (A)  \subseteq \Phi_\alpha (B)$. In particular, the intersection $\bigcap_{B \in \mathcal{B}_x} \Phi_\alpha(B)$ is nonempty and in fact consists of a single point that we define to be $\Omega_\alpha(x)$. We shall show that $\Omega_\alpha(x)$ is given by the expression stated in the proposition. It follows by \cref{lma: compatibility of the induced homomorphisms from an automorphism} that

\begin{align*}
	\mathsf{r}(\Omega_\alpha(x)) &= \mathsf{r}\Big(\bigcap_{B \in \mathcal{B}_x} \Phi_\alpha(B)\Big) \\
	&= \bigcap_{B \in \mathcal{B}_x} \mathsf{r}(\Phi_\alpha(B)) \\
	&= \bigcap_{B \in \mathcal{B}_x} \Upsilon_\alpha(r(B)) \\
	&= \Upsilon_\alpha \Big(\bigcap_{B \in \mathcal{B}_x} \mathsf{r}(B)\Big) = \Upsilon_\alpha(\mathsf{r}(x)),
\end{align*}
and likewise, $\mathsf{s}(\Omega_\alpha(x)) = \Upsilon_\alpha(\mathsf{s}(x))$. From this, it is clear that given any $B \in \mathcal{B}_x$, $\Phi_\alpha(B)$ contains $\Omega_\alpha(x)$, and $$ \Omega_\alpha(x) = \Upsilon_\alpha(\mathsf{r}(x)) \Phi_\alpha(B) = \Phi_\alpha(B)  \Upsilon_\alpha(s(x)) ,$$ which proves the claim.

\medskip

\textbf{Claim 2:} \emph{$\Omega_\alpha$ is a (topological) automorphism of $\calG$.}
Given $(x,y) \in \calG^{(2)}$ and bisections $A,B \in \Bco(\calG)$ such that $x \in A$ and $y \in B$, we 
use at the second step that~$\Phi_\alpha$ is multiplicative to get
\begin{align*}
	\Omega_\alpha(xy) &= \Phi_\alpha(AB) \Upsilon_\alpha (\sou(xy)) \\
	&= \Phi_\alpha(A) \Phi_\alpha(B) \Upsilon_\alpha (\sou(y)) \\
	&= \Phi_\alpha(A) \Upsilon_\alpha (\ran(y)) \Phi_\alpha(B) \Upsilon_\alpha (\sou(y)) \\
	&= \Phi_\alpha(A) \Upsilon_\alpha (\sou(x)) \Phi_\alpha(B) \Upsilon_\alpha (\sou(y)) \\
	&= \Omega_\alpha(x) \Omega_\alpha(y) .
\end{align*}
In order to check that $\Omega_\alpha$ is continuous, let $(x_j)_{j
\in J}$ be a net in $\calG$ converging to~$x$. 
Given $B \in \Bco(\calG)$ containing $x$, there exists $j_0$ such that $j \geq j_0$ implies $x_j \in B$. Thus, 
\[\lim_j \Omega_\alpha(x_j) = \lim_j \Phi_\alpha(B) \Upsilon_\alpha (\sou(x_j)) = \Phi_\alpha(B) \Upsilon_\alpha (\sou(x)) = \Omega_\alpha(x),
\]
and hence $\Omega_\alpha$ is continuous. Finally, for $x\in\calG$ we have
\begin{align*} \Omega_{\alpha^{-1}}(\Omega_\alpha(x)) &= \Omega_{\alpha^{-1}} \big(\Phi_\alpha(B) \Upsilon_\alpha (\sou(x))\big) \\
&= \Phi_\alpha^{-1} (\Phi_\alpha(B)) \Upsilon_\alpha^{-1} (\Upsilon_\alpha (\sou(x))) \\
&= B \sou(x) = x ,\end{align*}
showing that $\Omega_{\alpha^{-1}}$ is the inverse of $\Omega_\alpha$. 
This proves the claim.

\medskip

\textbf{Claim 3:} \emph{The resulting map $\Omega$ is a group homomorphism.}
To prove the claim, let $\alpha,\beta\in \aut(\frakA(\calG))$, let $x\in \calG$ and let $B\in\Bco(\calG)$ contain $x$. 
Using at the second step that $\Phi_\alpha(B)\in \Bco(\calG)$ contains $\Omega_\alpha(x)$ and $\mathsf{s}\big(\Omega_\alpha(x)\big)=\Upsilon_\alpha(\sou(x))$, and at the last step that both $\Phi$ and $\Upsilon$ are homomorphisms, we get
\[\Omega_\beta (\Omega_\alpha (x)) = \Omega_\beta \big(\Phi_\alpha(B) \Upsilon_\alpha(\sou(x))\big) = \Phi_\beta(\Phi_\alpha(B)) \Upsilon_\beta(\Upsilon_\alpha(\sou(x)))  = \Omega_{\beta \circ \alpha} (x) ,\] as desired.
\end{proof}

%==========================================================================================
The following is the second main result of this work, which should be compared
to part~(i) of \cite[Proposition~5.7]{Mat12HomTopFullGp}. 
Note that Matui needs to use the Weyl groupoid from \cite{Ren08Cartan} to produce some of the homomorphisms therein, because of which he only considers automorphisms preserving the diagonal. In contrast, we instead rely on the tight groupoid from \cite{Exe10ReconstrTotDiscGpds} and the results from \cite{GarPalThi25arX:RigidityAmpleGpds} to produce the homomorphisms, which means that we need not put any restrictions on the automorphism groups. (We do, however, need to assume effectiveness for the computation of the outer automorphism group, but this is to be expected.)

In the next result, for a Hausdorff, ample groupoid $\calG$ with compact unit space, we will naturally regard $\F(\calG)$ as a subgroup of $\aut(\calG)$ by identifying a full bisection $B\in \F(\calG)$ with the automorphism $\ad(B) (x) = BxB^{-1}$, for $x\in\calG$.

%==========================================================================================
\begin{thm} 
\label{prp:exact sequences of automorphism groups}
Let $\calG$ be a Hausdorff ample, effective groupoid, and let
$\frakA(\calG)$ denote either $L^I (\calG)$, $\Fpast(\calG)$, or $\Fp(\calG)$, with $p \in [1, \infty) \setminus \{2\}$. 
Then the canonical sequence 
\[\tag{\textbf{A}}
	\begin{tikzcd}
		1 \arrow{r}{} & \Zycles (\calG, \TT) \arrow{r}{\Gamma} & \aut(\frakA(\calG)) \arrow{r}{\Omega} & \aut(\calG) \arrow{r}{} & 1 
	\end{tikzcd}
	\]
is split exact.	If, in addition, $\calG$ is effective and has compact unit space, then there are induced split exact sequences
	\[\tag{\textbf{I}}
	\begin{tikzcd}
		1 \arrow{r}{} & \Bound (\calG, \TT) \arrow{r}{\Gamma} & \inn(\frakA(\calG)) \arrow{r}{\Omega} & \F(\calG) \arrow{r}{} & 1 ,
	\end{tikzcd}
	\]
	and 
	\[\tag{\textbf{O}}
	\begin{tikzcd}
		1 \arrow{r}{} & \Hquot (\calG, \TT) \arrow{r}{\overline{\Gamma}} & \out(\frakA(\calG)) \arrow{r}{\overline{\Omega}} & \aut(\calG)/ \F(\calG) \arrow{r}{} & 1.
	\end{tikzcd}
	\]
\end{thm}
\begin{proof}
We say a few words about what needs to be proved. For the sequence (\textbf{A}) we need to show exactness and also construct a split. For (\textbf{I}), we first need to show that the maps are well-defined (namely, $\Gamma$ maps the 1-coboundaries to inner automorphisms, and $\Omega$ maps the inner automorphisms to full bisections), and both exactness and splitting also need to be established. Finally, for (\textbf{O}) only split-exactness needs to be shown.
We again divide the proof into claims, following the above discussion. 

\medskip

\textbf{Claim 1:} \textit{There is a multiplicative section for 
	$\Omega\colon\aut(\frakA(\calG))\to\aut(\calG)$.}
Let $\theta \in \aut(\calG)$ be given. Functoriality of the construction of $\frakA(\calG)$ implies that there is an automorphism $\alpha_\theta\in\aut(\frakA(\calG))$
such that $\alpha_\theta(f)(x)=f(\theta^{-1}(x))$ for all $f\in C_c(\calG)$ and 
all $x\in\calG$. By construction we have $\Upsilon_{\alpha_\theta}=\theta$. Moreover, denoting by $\theta_\ast\in\aut(\Bco(\calG))$ the automorphism given by
$\theta_\ast(B) = \theta(B)$, for $B \in \Bco(\calG)$, we also get $\Phi_{\alpha_\theta}=\theta_\ast$.
It follows that $\Omega_{\alpha_\theta} (x) = \theta(B) \theta(\sou(x)) = \theta(B \sou(x)) = \theta(x)$, and hence $\Omega_{\alpha_\theta}=\theta$. Naturality of the 
construction implies the map 
$s\colon \aut(\calG)\to \aut(\frakA(\calG))$ given by $s(\theta)=\alpha_\theta$ for 
$\theta\in \aut(\calG)$, is a group homomorphism, as desired.

\medskip

\textbf{Claim 2:}  \textit{The sequence (\textbf{A}) is exact.} 
Note that Claim~1 implies that $\Omega$ is surjective.
It is clear from the construction of $\Gamma$ (see \cref{dfn:Gamma}) that it is injective. We now check exactness at $\aut(\frakA(\calG))$.
It follows from the formula for $\Omega_\alpha$ (see 
\cref{prop:Omega}) that if $\alpha \in \aut(\frakA(\calG))$ satisfies $\Phi_\alpha = \id_{\Bco(\calG)}$ and $\Upsilon_\alpha = \id_{\calG^{(0)}}$, then $\Omega_\alpha = \id_{\calG}$. 
Thus, in order to show that $\Omega\circ\Gamma=\id_{\calG}$, given $\xi \in \Zycles (\calG, \TT)$ it suffices to show that $\Phi_{\Gamma(\xi)}=\id_{\Bco(\calG)}$ and $\Upsilon_{\Gamma(\xi)}=\id_{\calG^{(0)}}$. 
To see that this is the case, note that for $f \in C_0 (\calG^{(0)})$ we have 
\[\Gamma(\xi) (f) (x) = \begin{cases} \xi (x) f(x) = f(x), & \mbox{ if } x \in \calG^{(0)},\\
0, & \mbox{ else.}
                \end{cases}\]
In other words, $\Gamma (\xi)$ is the identity on $C_0 (\calG^{(0)})$, and thus $\Upsilon_{\Gamma(\xi)} = \id_{\calG^{(0)}}$.
On the other hand, for $B\in\Bco(\calG)$, let $\xi_B\colon \calG^{(0)}\to \mathbb{C}$ be given by 
\[\xi_B(\ran(x)) = \begin{cases} \xi(x) , & \mbox{ if } x \in B,\\
	0, & \mbox{ else.}
\end{cases}\]
For a continuous function $f\colon \ran(B)\to \mathbb{T}$, recall that
$f \ast \mathbbm{1}_B$ belongs to $\PIMP(\frakA(\calG))$. In this case, we have \[\Gamma(\xi) (f \ast \mathbbm{1}_B) = (f \cdot \xi_B) \ast \mathbbm{1}_B.\] 
It follows then that \[\sigma(\Gamma(\xi) (\mathbbm{1}_B)) = \sigma(\xi_B \ast \mathbbm{1}_B) = B ,\] and hence 
$\Phi_{\Gamma(\xi)}=\id_{\Bco(\calG)}$ for all $B \in \Bco(\calG)$. We conclude that $\Omega\circ\Gamma=\id_{\calG}$. 
	
Let us see next that $\ker(\Omega) \subseteq \mathrm{im}(\Gamma)$. Let $\alpha \in \aut(\frakA(\calG))$ satisfy $\Omega_\alpha = \id_\calG$. 
The definition of $\Omega$ (see \cref{prop:Omega}) shows that $\Omega_\alpha(u) = u$ for any $u \in \calG^{(0)}$, and therefore $\alpha|_{C_0 (\calG^{(0)})} = \id_{C_0 (\calG^{(0)})}$. In order to show that $\alpha$ is in the image of $\Gamma$,
we define a map $\xi \colon \calG \to \TT$ as follows: given $x \in \calG$, find any $A \in \Bco(\calG)$ such that $x \in A$, and set $\xi(x) = \alpha (\mathbbm{1}_A)(x)$. 

Let us check that $\xi$ is well defined; we begin by showing that it is independent
of~$A$.
First, it is not hard to check that $\Omega_\alpha = \id_\calG$ implies $\Phi_\alpha = \id_{\Bco(\calG)}$. By the definition of $\Phi_\alpha$ (see \cref{dfn:Phi}), it follows that for every $A\in\Bco(\calG)$ there exists $f_A \in C(\ran(B), \TT)$ such that $\alpha(\mathbbm{1}_A) = f_A \ast \mathbbm{1}_A$. Note that
\[\tag{3.1}\alpha(\mathbbm{1}_A) (x) = f_{A}(\ran(x)) \in \TT,\] 
and our task is to show that the above expression does not depend on the compact open bisection $A$ containing $x$.
Suppose now that $A,B \in \Bco(\calG)$ both contain $x$. Then 
\[\mathbbm{1}_{\ran(A \cap B) } \mathbbm{1}_{B} \mathbbm{1}_{A^{-1}} \mathbbm{1}_{\ran(A \cap B)} = \mathbbm{1}_{\ran(A \cap B) } \mathbbm{1}_{A \cap B } \mathbbm{1}_{(A \cap B)^{-1}} \mathbbm{1}_{\ran(A \cap B) } \] 
belongs to $C_0 (\calG^{(0)})$, so that 
\begin{align*}
1 &= \mathbbm{1}_{\ran(A \cap B) } \mathbbm{1}_{B} \mathbbm{1}_{A^{-1}} \mathbbm{1}_{\ran(A \cap B)} (\ran(x)) \\
&= \alpha(\mathbbm{1}_{\ran(A \cap B) } \mathbbm{1}_{B} \mathbbm{1}_{A^{-1}} \mathbbm{1}_{\ran(A \cap B)}) (\ran(x)) \\
&= \alpha(\mathbbm{1}_{\ran(A \cap B) } \mathbbm{1}_{B} \mathbbm{1}_{\sou(A \cap B)}) (x) \alpha(\mathbbm{1}_{\ran(A \cap B)} \mathbbm{1}_{A} \mathbbm{1}_{\sou(A \cap B)} )^{\dag} (x^{-1}) \\
&= f_{B}(\ran(x)) \overline{f_{A} (\ran(x))} ,
\end{align*}
	which implies that $$ f_{A} (\ran(x)) = f_{B} (\ran(x)) ,$$ as desired. Thus $\xi$ is a well-defined function, and its range is contained in $\TT$ by (3.1). 
Continuity of the maps involved implies that $\xi$ is continuous. 
To see that~$\xi$ is a homomorphism, fix a composable pair $(x,y) \in \calG^{(2)}$. Let $A,B \in \Bco(\calG)$ satisfy $x \in A$ and $y \in B$. Then $xy \in AB$ with $AB \in \Bco(\calG)$, and
	\begin{align*}
			\xi(xy) &= \alpha(\mathbbm{1}_{AB})(xy) \\
			&= \alpha(\mathbbm{1}_A \ast \mathbbm{1}_B) (xy) \\
			& = \alpha(\mathbbm{1}_A) \ast \alpha(\mathbbm{1}_B) (xy) \\
			&= \alpha(\mathbbm{1}_A)(x) \alpha(\mathbbm{1}_B)(y) = \xi(x) \xi(y),
	\end{align*}
	which shows that $\xi$ is indeed a homomorphism.

Given $A\in \Bco(\calG)$ and a continuous function $f \colon \ran(A) \to \CC$, one can readily check that $\Gamma(\xi) (f \ast \mathbbm{1}_A ) = \alpha (f \ast \mathbbm{1}_A)$. Since the span of the elements of the form $f\ast\mathbbm{1}_A$ is dense in $\frakA(\calG)$, continuity of $\Gamma(\xi)$ and $\alpha$ imply that $\Gamma(\xi) = \alpha$. We conclude that $\ker(\Omega) \subseteq \mathrm{im}(\Gamma)$, showing that the sequence is indeed exact.

\medskip

Let us, from now on and until the end of the proof, suppose that $\calG$ is effective and has compact unit space.

\medskip

\textbf{Claim 3:} \textit{The sequence \textbf{(I)} is well-defined.}
Recall that we are identifying $\F(\calG)$ with a subgroup of $\aut(\calG)$ as explained before the statement of this theorem. Recall also that any invertible isometry in $\frakA(\calG)$ has the form $ f \ast \mathbbm{1}_B$, for some $f \in C(\calG^{(0)}, \TT)$ and $B \in \F(\calG)$; see \cref{prp:UFpGpd}. 
We need to show that $\Gamma(\Bound(\calG,\TT))\subseteq \inn(\frakA(\calG))$ and 
that $\Omega(\inn(\frakA(\calG)))\subseteq \F(\calG)$. 

For the first of these, let $f \in C(\calG^{(0)}, \TT)\subseteq \invisom(\frakA(\calG))$ be given, and recall that we write $\hat{f} \in B^1 (\calG, \TT)$ for the 1-coboundary given by $\hat{f}(x) = \overline{f(\ran(x))} f(\sou(x))$ for $x \in \calG$. For $g \in C_c (\calG)$ and $x\in \calG$, we have 
\[ \Gamma(\hat{f}) (g) (x) = \hat{f}(x) g(x) = \overline{f(\ran(x))} f(\sou(x)) g(x) = (\overline{f} \ast g \ast f) (x) = \ad (f) (g) (x).\] 
By density and continuity, we deduce that $\Gamma(\hat{f}) = \ad(f)$. Thus, $\Gamma$ maps $\Bound (\calG, \TT)$ into $\inn(\frakA(\calG))$. 

Now, to check that $\Omega(\inn(\frakA(\calG)))\subseteq \F(\calG)$, let $\alpha \in \inn(\frakA(\calG))$ be given. By \cref{prp:UFpGpd} there exist $B \in \F(\calG)$ and $f \in C(\calG^{(0)}, \TT)$ so that $ \alpha = \ad(f \ast \mathbbm{1}_B)$. Then for any $A \in \Bco(\calG)$, we have $$ \alpha (\mathbbm{1}_A) = \mathbbm{1}_{B^{-1}} \ast \overline{f} \ast \mathbbm{1}_A \ast f \ast \mathbbm{1}_B = (\overline{f} \circ \rho_B) \ast \mathbbm{1}_{B^{-1}AB} \ast (f \circ \rho_B) .$$ Therefore, $\Phi_{\alpha} (A) = B^{-1}AB$, and it is then not hard to see that $\Upsilon_{\alpha} = \rho_{B}^{-1}$. Now, for $x \in \calG$ and $A \in \Bco(\calG)$ with $x \in A$, we have $$\Omega_{\alpha} (x) = B^{-1}AB \rho_{B}^{-1} (\sou(x)) = B^{-1}AB B^{-1} \sou(x) B = B^{-1}A \sou(x) B = B^{-1} x B ,$$ and hence we conclude that $\Omega_{\alpha} = \ad(B) \in \F(\calG)$, as desired.

\medskip

\textbf{Claim 4:} \textit{The sequence (\textbf{I}) is split exact.}
We begin by noting that the section $s\colon \aut(\calG)\to\aut(\frakA(\calG))$ constructed in Claim~1 satisfies $s(\F(\calG))\subseteq \inn(\frakA(\calG))$; indeed,
it is routine to check that for $B\in\F(\calG)$, we have 
$s(\ad(B))=\ad(\mathbbm{1}_B)\in\inn(\frakA(\calG))$. 

Next, we show exactness. The restriction of $\Gamma$ is injective because $\Gamma$ is, and we have $\Gamma(\Bound(\calG,\TT)) \subseteq \ker(\Omega)$ by Claim~2. Let us show next that $\ker(\Omega) \subseteq \mathrm{im}(\Gamma)$. Let $\alpha\in\aut(\frakA(\calG))$ and assume that $\Omega_\alpha = 1_{\aut(\calG)}$. Choose $f \in C(\calG^{(0)}, \TT)$ and $A \in \F(\calG)$ with $\alpha =  \ad (f \ast \mathbbm{1}_A)$. Given $B \in \Bco(\calG)$, recall from Claim~3 that $\Phi_\alpha (B) = A^{-1} B A$, and $\Upsilon_\alpha = \rho_{A}^{-1}$. Thus, 
\[ x = \Omega_\alpha (x) = A^{-1}BA \rho_{A}^{-1} (\sou(x)) = A^{-1}BA A^{-1} \sou(x)A = A^{-1}B \sou(x) A = A^{-1}xA ,\] 
or equivalently 
$ Ax = xA,$ for all $x \in \calG$. In particular, $Au = uA$ for all $u \in \calG^{(0)}$, and thus $A \subseteq \rm{Iso}(\calG)^{\circ}$. Since $\calG$ is effective, we have $\rm{Iso}(\calG)^{\circ} = \calG^{(0)}$, and hence $B = \calG^{(0)}$. We conclude that $\alpha = \ad(f) = \Gamma(\hat{f})$, showing that $\ker(\Omega) \subseteq \rm{im}(\Gamma)$, and proving that the sequence (\textbf{I}) is split exact. 

\medskip
	
\textbf{Claim 5:} \textit{The sequence \textbf{(O)} is split exact.}
It follows from exactness of \textbf{(A)} and \textbf{(I)} that \textbf{(O)} is well defined. We also have $\rm{im}(\overline{\Gamma}) \subseteq \ker(\overline{\Omega})$. Moreover, if $s\colon \aut(\calG)\to\aut(\frakA(\calG))$ is the section for $\Omega$ constructed in Claim~1, which was shown in Claim~4 to be a section for $\inn(\frakA(\calG))\to \F(\calG)$, then $s$ canonically induces a well defined map $\overline{s}\colon \aut(\calG)/ \F(\calG) \to \out(\frakA(\calG))$ that satisfies $\overline{\Omega} \circ \overline{s} = \id_{\aut(\calG)/ \F(\calG)}$, so that \textbf{(O)} splits. 

In order to see that $\overline{\Gamma}$ is injective, let $\xi\in \Zycles(\calG,\TT)$ satisfy $\Gamma (\xi) \in \inn(\frakA(\calG))$. Choose $f \in C(\calG^{(0)}, \TT)$ and $B \in \F(\calG)$ with $\Gamma(\xi) = \ad(f \ast \mathbbm{1}_B)$. We will show that $\xi=\widehat{f}$. Given $A \in \Bco(\calG)$ and $x \in \calG$, we have that
\begin{align*}  \xi(x)  \mathbbm{1}_A(x) &=\big(\Gamma(\xi)  \mathbbm{1}_A\big)(x) \\
	&= \ad(f\ast\mathbbm{1}_B)( \mathbbm{1}_A)(x)\\
	&= (\overline{f}\circ\rho_B) \ast \mathbbm{1}_{B^{-1}AB} \ast (f \circ \rho_B) (x).
\end{align*} 
Since $\xi$ and $\overline{f} \circ \rho_B (\mathsf{r}(x)) f \circ \rho_B (\mathsf{s}(x))$ take values in $\TT$, we may conclude that $BA = AB$ for all $A \in \Bco(\calG)$, from which it follows that $Bu = uB$ for all $u\in \calG^{(0)}$. Hence $B \subseteq \rm{Iso}(\calG)^{\circ} = \calG^{(0)}$, which implies that $B = \calG^{(0)}$. We conclude that $\Gamma (\xi) = \ad(f) = \Gamma (\hat{f})$, so that $\xi = \hat{f}$ by injectivity of $\Gamma$, proving in turn injectivity of $\overline{\Gamma}$. 

Finally, let us show that $\ker(\overline{\Omega}) \subseteq \rm{im}(\overline{\Gamma})$. Recall that $s(\ad(B)) = \ad(\mathbbm{1}_B)$, and that $\Omega_{\ad(\mathbbm{1}_B)} = \ad(B)$, for every $B\in\F(\calG)$.
Using this, we see that if $\alpha \in \aut(\frakA(\calG))$ satisfies $\Omega_\alpha = \ad(B)$ for some $B \in \F(\calG)$, then 
\[\ad(\mathbbm{1}_B)^{-1}\circ \alpha \in \ker(\Omega) = \rm{im}(\Gamma) ,\] so that there exists $\xi \in \Zycles (\calG, \TT)$ satisfying $\Gamma(\xi) = \ad(\mathbbm{1}_B)^{-1} \circ \alpha$. In this case, since
$\alpha \circ \Gamma(\xi)^{-1} = \ad(\mathbbm{1}_B)$ belongs to $\inn(\frakA(\calG))$, we get
\[\alpha \cdot \inn(\frakA(\calG)) = \Gamma(\xi)\cdot \inn(\frakA(\calG)) .\] 
In other words, $\overline{\Gamma}(\xi \cdot \Bound (\calG , \TT)) = \alpha \cdot \inn(\frakA(\calG))$, showing that $\ker(\overline{\Omega}) \subseteq \rm{im}(\overline{\Gamma})$, which finishes the proof that \textbf{(O)} is split exact.
\end{proof}

%\bibliographystyle{aomalphaMyShort}
%\bibliography{References}

\begin{thebibliography}{BKM25}
	
	\bibitem[AO22]{AusOrt22GpdsHermBAlg}
	\bgroup\scshape{}A.~Austad\egroup{} and \bgroup\scshape{}E.~Ortega\egroup{},
	Groupoids and {H}ermitian {B}anach {$*$}-algebras,  \emph{Internat. J. Math.}
	\textbf{33} (2022), Paper No.~2250090, 25p.
	
	\bibitem[BKM25]{BarKwaMck25BAlgTwGpdInvSgpDisintLp}
	\bgroup\scshape{}K.~Bardadyn\egroup{},
	\bgroup\scshape{}B.~Kwa\'{s}niewski\egroup{}, and
	\bgroup\scshape{}A.~McKee\egroup{}, Banach algebras associated to twisted
	\'{e}tale groupoids: inverse semigroup disintegration and representations on
	{$L^p$}-spaces,  \emph{J. Funct. Anal.} \textbf{289} (2025), Paper No.
	111163, 66p.
	
	\bibitem[BK23]{BarKwa23arX:TopFreeActionsIdlsTwBAlgCrProd}
	\bgroup\scshape{}K.~Bardadyn\egroup{} and \bgroup\scshape{}B.~K.
	Kwaśniewski\egroup{}, Topologically free actions and ideals in twisted
	{B}anach algebra crossed products, preprint (arXiv:2307.01685), 2023.
	
	\bibitem[Bri04]{Bri04HigherDimThompsonGps}
	\bgroup\scshape{}M.~G. Brin\egroup{}, Higher dimensional {T}hompson groups,
	\emph{Geom. Dedicata} \textbf{108} (2004), 163--192.
	
	\bibitem[BS19]{BriSca19CSimpleReprTopFullGps}
	\bgroup\scshape{}K.~A. Brix\egroup{} and \bgroup\scshape{}E.~Scarparo\egroup{},
	{${\rm C}^*$}-simplicity and representations of topological full groups of
	groupoids,  \emph{J. Funct. Anal.} \textbf{277} (2019), 2981--2996.
	
	\bibitem[CGT24]{ChoGarThi24LpRigidity}
	\bgroup\scshape{}Y.~Choi\egroup{}, \bgroup\scshape{}E.~Gardella\egroup{}, and
	\bgroup\scshape{}H.~Thiel\egroup{}, Rigidity results for {$L^p$}-operator
	algebras and applications,  \emph{Adv. Math.} \textbf{452} (2024), Paper No.
	109747, 47~pages.
	
	\bibitem[CJN20]{ChoJusNek20TopFullGpsZdActions}
	\bgroup\scshape{}M.~Chornyi\egroup{}, \bgroup\scshape{}K.~Juschenko\egroup{},
	and \bgroup\scshape{}V.~Nekrashevych\egroup{}, On topological full groups of
	{$\Bbb Z^d$}-actions,  \emph{Groups Geom. Dyn.} \textbf{14} (2020), 61--79.
	
	\bibitem[Elk25]{Elk25SymmPseudofctLp}
	\bgroup\scshape{}E.~M. Elki{\ae}r\egroup{}, Symmetrized pseudofunction algebras
	from {$L^p$}-representations and amenability of locally compact groups,
	\emph{Expo. Math.} \textbf{43} (2025), Paper No. 125685, 20~pages.
	
	\bibitem[EP23]{ElkPoo23arX:PropTBAlg}
	\bgroup\scshape{}E.~M. Elki{\ae}r\egroup{} and
	\bgroup\scshape{}S.~Pooya\egroup{}, Property ({T}) for {B}anach algebras, J.
	Operator Theory (to appear), preprint (arXiv:2310.18136), 2023.
	
	\bibitem[Exe10]{Exe10ReconstrTotDiscGpds}
	\bgroup\scshape{}R.~Exel\egroup{}, Reconstructing a totally disconnected
	groupoid from its ample semigroup,  \emph{Proc. Amer. Math. Soc.}
	\textbf{138} (2010), 2991--3001.
	
	\bibitem[GPT25]{GarPalThi25arX:RigidityAmpleGpds}
	\bgroup\scshape{}E.~Gardella\egroup{},
	\bgroup\scshape{}M.~Palmstr{\o}m\egroup{}, and
	\bgroup\scshape{}H.~Thiel\egroup{}, Rigidity of pseudofunction algebras of
	ample groupoids, preprint (arXiv:2506.09563), 2025.
	
	\bibitem[GT23]{GarTan23arX:GenThompsonGpV}
	\bgroup\scshape{}E.~Gardella\egroup{} and \bgroup\scshape{}O.~Tanner\egroup{},
	Generalisations of {T}hompson's group {V} arising from purely infinite
	groupoids, preprint (arXiv:2302.04078), 2023.
	
	\bibitem[GG24]{GarGun24pre:EmbLpOpAlg}
	\bgroup\scshape{}E.~Gardella\egroup{} and
	\bgroup\scshape{}J.~Gundelach\egroup{}, Embeddings of {$L^p$}-operator
	algebras, in preparation, 2024.
	
	\bibitem[GL17]{GarLup17ReprGrpdLp}
	\bgroup\scshape{}E.~Gardella\egroup{} and \bgroup\scshape{}M.~Lupini\egroup{},
	Representations of \'{e}tale groupoids on {$L^p$}-spaces,  \emph{Adv. Math.}
	\textbf{318} (2017), 233--278.
	
	\bibitem[GT22]{GarThi22IsoConv}
	\bgroup\scshape{}E.~Gardella\egroup{} and \bgroup\scshape{}H.~Thiel\egroup{},
	Isomorphisms of algebras of convolution operators,  \emph{Ann. Sci. \'{E}c.
		Norm. Sup\'{e}r. (4)} \textbf{55} (2022), 1433--1471.
	
	\bibitem[GPS99]{GioPutSka99FullGpsCMS}
	\bgroup\scshape{}T.~Giordano\egroup{}, \bgroup\scshape{}I.~F. Putnam\egroup{},
	and \bgroup\scshape{}C.~F. Skau\egroup{}, Full groups of {C}antor minimal
	systems,  \emph{Israel J. Math.} \textbf{111} (1999), 285--320.
	
	\bibitem[Hah78]{Hah78RegReprMsrGpds}
	\bgroup\scshape{}P.~Hahn\egroup{}, The regular representations of measure
	groupoids,  \emph{Trans. Amer. Math. Soc.} \textbf{242} (1978), 35--72.
	
	\bibitem[JM13]{JusMon13CantorSystems}
	\bgroup\scshape{}K.~Juschenko\egroup{} and \bgroup\scshape{}N.~Monod\egroup{},
	Cantor systems, piecewise translations and simple amenable groups,
	\emph{Ann. of Math. (2)} \textbf{178} (2013), 775--787.
	
	\bibitem[Kri80]{Kri80DimClassHomeoGps}
	\bgroup\scshape{}W.~Krieger\egroup{}, On a dimension for a class of
	homeomorphism groups,  \emph{Math. Ann.} \textbf{252} (1979/80), 87--95.
	
	\bibitem[Li25]{Li25AmpleGpdsTopFullGpsAlgKThyInfLoopsSpaces}
	\bgroup\scshape{}X.~Li\egroup{}, Ample groupoids, topological full groups,
	algebraic {K}-theory spectra and infinite loop spaces,  \emph{Forum Math. Pi}
	\textbf{13} (2025), Paper No. e9, 56~pages.
	
	\bibitem[Mat06]{Mat06RmksTopFullGpsCMS}
	\bgroup\scshape{}H.~Matui\egroup{}, Some remarks on topological full groups of
	{C}antor minimal systems,  \emph{Internat. J. Math.} \textbf{17} (2006),
	231--251.
	
	\bibitem[Mat12]{Mat12HomTopFullGp}
	\bgroup\scshape{}H.~Matui\egroup{}, Homology and topological full groups of
	\'{e}tale groupoids on totally disconnected spaces,  \emph{Proc. Lond. Math.
		Soc. (3)} \textbf{104} (2012), 27--56.
	
	\bibitem[Mat15]{Mat15TopFullGpOneSidedShift}
	\bgroup\scshape{}H.~Matui\egroup{}, Topological full groups of one-sided shifts
	of finite type,  \emph{J. Reine Angew. Math.} \textbf{705} (2015), 35--84.
	
	\bibitem[Nek19]{Nek19SimpleGpsDynOrigin}
	\bgroup\scshape{}V.~Nekrashevych\egroup{}, Simple groups of dynamical origin,
	\emph{Ergodic Theory Dynam. Systems} \textbf{39} (2019), 707--732.
	
	\bibitem[Nek04]{Nek04CuntzPimsnerAlgsGpActions}
	\bgroup\scshape{}V.~V. Nekrashevych\egroup{}, Cuntz-{P}imsner algebras of group
	actions,  \emph{J. Operator Theory} \textbf{52} (2004), 223--249.
		
	\bibitem[Ng06]{Ng_amenability_2006}
	\bgroup\scshape{}P. Ng\egroup{}, Amenability of the sequence of unitary groups associated with a C*-algebra,  \emph{Indiana Univ. Math. J.} \textbf{55} (2006), 1389--1400.
	
	\bibitem[NO21a]{NylOrt21KEPGpdsAHConj}
	\bgroup\scshape{}P.~Nyland\egroup{} and \bgroup\scshape{}E.~Ortega\egroup{},
	Katsura-{E}xel-{P}ardo groupoids and the {AH} conjecture,  \emph{J. Lond.
		Math. Soc. (2)} \textbf{104} (2021), 2240--2259.
	
	\bibitem[NO21b]{NylOrt21MatuiAHConjGraphGpds}
	\bgroup\scshape{}P.~Nyland\egroup{} and \bgroup\scshape{}E.~Ortega\egroup{},
	Matui's {AH} conjecture for graph groupoids,  \emph{Doc. Math.} \textbf{26}
	(2021), 1679--1727.
	
	\bibitem[NO19]{NylOrt19TopFullGpsAmpleGpds}
	\bgroup\scshape{}P.~Nyland\egroup{} and \bgroup\scshape{}E.~Ortega\egroup{},
	Topological full groups of ample groupoids with applications to graph
	algebras,  \emph{Internat. J. Math.} \textbf{30} (2019), 1950018, 66~pages.
	
	\bibitem[Phi12]{Phi12arX:LpAnalogsCtz}
	\bgroup\scshape{}N.~C. Phillips\egroup{}, Analogs of {C}untz algebras on
	{$L^p$} spaces, preprint (arXiv:1201.4196 [math.FA]), 2012.
	
	\bibitem[Ren80]{Ren80GrpdApproach}
	\bgroup\scshape{}J.~Renault\egroup{}, \emph{A groupoid approach to \ca{s}},
	\emph{Lecture Notes in Mathematics} \textbf{793}, Springer, Berlin, 1980.
	
	\bibitem[Ren08]{Ren08Cartan}
	\bgroup\scshape{}J.~Renault\egroup{}, Cartan subalgebras in \ca{s},
	\emph{Irish Math. Soc. Bull.} (2008), 29--63.
	
	\bibitem[Ric67]{Ric67AmenGpsFixedPtPty}
	\bgroup\scshape{}N.~W. Rickert\egroup{}, Amenable groups and groups with the
	fixed point property,  \emph{Trans. Amer. Math. Soc.} \textbf{127} (1967),
	221--232.
	
	\bibitem[SW20]{SamWie20QuasiHermAmen}
	\bgroup\scshape{}E.~Samei\egroup{} and \bgroup\scshape{}M.~Wiersma\egroup{},
	Quasi-{H}ermitian locally compact groups are amenable,  \emph{Adv. Math.}
	\textbf{359} (2020), 106897, 25~pages.
	
	\bibitem[SW24]{SamWie24ExoticCAlgsGeomGps}
	\bgroup\scshape{}E.~Samei\egroup{} and \bgroup\scshape{}M.~Wiersma\egroup{},
	Exotic \ca{s} of geometric groups,  \emph{J. Funct. Anal.} \textbf{286}
	(2024), Paper No. 110228, 32~pages.
	
	\bibitem[Sca23]{Sca23DichotomyTopFullGps}
	\bgroup\scshape{}E.~Scarparo\egroup{}, A dichotomy for topological full groups,
	\emph{Canad. Math. Bull.} \textbf{66} (2023), 610--616.
	
	\bibitem[Tan23]{Tan23arX:SteinGpsTopFullGps}
	\bgroup\scshape{}O.~Tanner\egroup{}, Studying {S}tein's groups as topological
	full groups, preprint (arXiv:2312.07375), 2023.
	
\end{thebibliography}

\providecommand{\bysame}{\leavevmode\hbox to3em{\hrulefill}\thinspace}
\providecommand{\noopsort}[1]{}
\providecommand{\mr}[1]{\href{http://www.ams.org/mathscinet-getitem?mr=#1}{MR~#1}}
\providecommand{\zbl}[1]{\href{http://www.zentralblatt-math.org/zmath/en/search/?q=an:#1}{Zbl~#1}}
\providecommand{\jfm}[1]{\href{http://www.emis.de/cgi-bin/JFM-item?#1}{JFM~#1}}
\providecommand{\arxiv}[1]{\href{http://www.arxiv.org/abs/#1}{arXiv~#1}}
\providecommand{\doi}[1]{\url{http://dx.doi.org/#1}}
\providecommand{\MR}{\relax\ifhmode\unskip\space\fi MR }
\providecommand{\MRhref}[2]{%
	\href{http://www.ams.org/mathscinet-getitem?mr=#1}{#2}
}
\providecommand{\href}[2]{#2}

\end{document}